\documentclass[12pt]{amsart}
\usepackage{latexsym,amssymb,amsfonts,amsmath, amscd, euscript, MnSymbol}
\usepackage{pdfsync}
\usepackage{enumerate}
\usepackage[pdftex]{graphicx}
\usepackage{wrapfig}
\usepackage{epsfig}
\usepackage{pgf,tikz}
\usetikzlibrary{arrows,automata}
\usetikzlibrary{positioning,calc}
\usepackage{amssymb, amsmath, amsthm, enumerate}
\usepackage{mathrsfs}
\usepackage{hyperref}
\usepackage{epstopdf}
\usepackage{enumerate}
\usepackage{verbatim}
\usepackage{float}
\usepackage{layout}
\usepackage[top=3cm, bottom=3cm, left=2.5cm, right=2.5cm]{geometry}
\usepackage{mathtools}

\newtheorem{theorem}{Theorem}[section]

\newtheorem{definition}[theorem]{Definition}
\newtheorem{example}[theorem]{Example}
\newtheorem{lemma}{Lemma}[section]

\newtheorem{problem}[theorem]{Problem}
\usepackage{siunitx}
\usepackage{graphics,mwe}% for image

\newsavebox{\hold}
\newlength{\holdht}

\begin{document}

\title[On Caristi fixed point theorem for set-valued mappings]{On Caristi fixed point theorem for set-valued mappings}
\author[K. Chaira, S. Chaira and S. Lazaiz]{K. Chaira$^1$, S. Chaira$^2$ and S. Lazaiz$^{3}$}
\address{$^1$L3A Laboratory\\ Department of Mathematics and Computer Sciences\\ Faculty of Sciences Ben M'sik\\ University of Hassan II Casablanca\\Morocco.}
\email{chaira\_karim@yahoo.fr}
\address{$^2$Laboratory of Mathematics and Applications\\Faculty of Sciences and Technologies Mohammedia\\ University Hassan II Casablanca\\ Morocco.}
\email{soumia\_chaira@yahoo.fr}
\address{$^3$LaSMA Laboratory\\ Department of Mathematics\\ Faculty of Sciences Dhar El Mahraz\\ University Sidi Mohamed Ben Abdellah, Fes, Morocco.}
\email{samih.lazaiz@usmba.ac.ma}


\begin{abstract}
The aim of this paper is to discuss Penot's problem on a generalization of Caristi's fixed point theorem. We settle this problem in the negative and we present some new theorems on the existence of fixed points of set-valued mappings in ordered metric spaces and reflexive Banach spaces.
\end{abstract}

\maketitle

\section{Introduction}

Caristi fixed point theorem is known as one of the most important results in metric fixed point theory \cite{ref2}. It is not only a generalization of the Banach contraction principle \cite{ref3} but it has also been proven to be equivalent to metric completeness \cite[Theorem 6]{ref4}. Moreover, it has been the subject of various generalizations and extensions (see e.g., \cite{ref6,ref5,ref7} and the related references therein). For instance, in attempting to generalize Caristi's fixed point theorem, Kirk \cite{ref1} raised the problem of whether a self-mapping $T$ has a fixed point on a metric space $(M,d)$ such that for all $x \in M$ 
$$\eta( d ( x , Tx )) \leq \phi( x ) -\phi( Tx ),$$
where $\eta$ is a function from $\mathbb{R}_+$, the set of all nonnegative reals, into $\mathbb{R}_+$, having appropriate properties. This problem has been settled in the negative by Khamsi in \cite{ref1}. However, in order to generalize Caristi's fixed point theorem many works have been made in the setting of set-valued mappings (cf. \cite{ref16,ref8,ref17,ref18,ref9}). In particular,  Penot in \cite{ref9} asked the following problem.

\begin{problem}[Penot 1979]
  Let $(M,d)$ be a complete metric space and $\phi : M\rightarrow \mathbb{R}_+$ a lower semi-continuous mapping. Let $T:M\rightarrow \mathcal{C}(M)$ be a set-valued mapping with nonempty closed values such that 
  \begin{equation}
    dist(x,T(x))\leq \phi(x)-\inf \phi(T(x)),\quad \forall x\in M.
  \end{equation}
  Does $T$ have a fixed point?
\end{problem}

Penot pointed out that the problem has a positive answer in the special case when $\phi(x)=(1-k)^{-1}dist(x,T(x))$ and $T$ has a closed graph. This is known as Nadler's fixed point theorem \cite{ref10}.

\medskip
The rest of this paper is organized as follows. In section \ref{sec2} some notions and notations are introduced and an example is given to answer Penot's problem. In attempt to improve Caristi's fixed point theorem, section \ref{sec3} concerns with the existence of fixed point for set-valued mappings in the spirit of Penot's formulation. In section \ref{sec4}, we obtain Caristi-type generalization in the framework of reflexive Banach spaces.

%%%%%%%%%%%%%%%%%%%%%%%%%%%%%%%%%%%%%%%%%%%%%%%%%%%%%%%%%%%%%%%%%%%%%%%%%%%%%%%%%%%%%%%%%%%%%%%%%%%%%%%%%%%%%%%%%%%%%%%%%%%%%%%%%%%%%%%%%%%%%%
\section{Preliminaries}\label{sec2}

We begin by recalling Caristi-type fixed point theorem for set-valued mappings. This theorem is due to Khamsi.

\begin{theorem}\cite[Theorem 4]{ref1}\label{Pthm}
Let $M$ be a complete metric space. Let $T : M \rightarrow 2^M$ be a set-valued map such that $T( x )$ is not empty and for all $x \in M$ there exists $y \in T ( x )$ such that
$$d ( x , y) \leq \phi( x ) -\phi( y ),$$
where $\phi$ is lower semi-continuous. Then $T$ has a fixed point, i.e., there exists $x \in M$ such that $x\in T ( x )$.
\end{theorem}  

\medskip
First, let us give an example to answer Penot's problem in the negative.

\begin{example}
Let $M=\mathbb{R}_+$ and set for each $x\in M$ 
$$\phi(x)=\frac{1}{x+1}$$
and let 
$$T(x)=[x+\phi(x),\infty).$$

Then $\phi$ is lower semi-continuous and $T$ has non-empty closed values. Moreover, we have for all $x\in M,$ 
$$\inf\;\phi(T(x))=0.$$
In addition, $y\in T(x)$ if and only if $y-x\geq \phi(x),$
thus 
$$dist(x,T(x))=\inf_{y\in T(x)}|x-y|=\phi(x)$$ which implies that for all $x\in M,$
$$ dist(x,T(x))\leq \phi(x)-\inf \phi(T(x)).$$
It is clear that $T$ does not have any fixed point in $M$.
\end{example}

Though the above example gives a negative answer to Penot's problem, one can ask when replacing $\mathcal{C}(M)$ by $\mathcal{K}(M)$ the family of all nonempty compact  subsets of $M$, does $T$ have a fixed point? The simple example below gives a negative answer.

\begin{example}
Let $M=[1,+\infty)$ and set for all $x\in M$ 
$$T(x)=[x+\frac{1}{x(x+1)},x+1]\quad\text{and} \quad \phi(x)=\frac{1}{x}.$$
Obviously we have 
$$ dist(x,T(x))\leq \phi(x)-\inf \phi(T(x)),$$
but $T$ does not have any fixed point in $M$.
\end{example} 

Recently, a new approach has been discovered dealing with the mix between order ideas and metric ideas. It allows many authors to get new extensions of numerous classical fixed point results. However, in order to study and discuss Penot's problem, we use the ideas of this direction. For further details, one can consult \cite{ref6,ref13,ref7,ref12,ref11} and references therein.

\medskip
Let $(M,d,\preceq)$ be a metric space endowed with a partial order $\preceq$. Throughout, we assume that the
order intervals are closed. Recall that an order interval is any of the subsets
$$ [a, \rightarrow) = \{ x \in M: a \preceq x \}
,\; 
(\leftarrow,a] = \{ x \in M: x \preceq a \} $$
for any $a \in M$. As a direct consequence of this, the subset
$$ [a,b] = \{ x \in M: a \preceq x \preceq b \} = [a,\rightarrow) \cap (\rightarrow,b]$$
is also closed for any $a, b \in M$.
We will say that $x, y \in M$ are comparable whenever $x \preceq y$ or $y\preceq x$.

\begin{definition}\rm
A sequence $(x_n)_{n\in \mathbb{N}}$ in a partially ordered set $(M,\preceq)$ is said to be:
\begin{enumerate}
\item[(i)] monotone increasing if $x_n \preceq x_{n+1}$ , for all $n \in \mathbb{N}$;
\item[(ii)] monotone decreasing if $x_{n+1}\preceq x_n$, for all $n \in \mathbb{N}$;
\item[(iii)] monotone sequence if it is either monotone increasing or decreasing.
\end{enumerate}
\end{definition}

In \cite{ref14}, the authors introduced the notion of monotone norm in Banach spaces and gave some interesting class of spaces with this property. In fact, they gave this definition for the corresponding distance (i.e. $d_{\|\cdot\|}(x,y)=\|x-y\|$). For that, we recall it here for metric spaces.

\begin{definition}
Let $(M,d,\preceq)$ be an ordered metric space, $d$ is said to be monotone if $x\preceq y \preceq z$ implies
$$\max\{d(x,y), d(y,z)\} \leq d(x,z)$$ 
for any $x,y,z\in M$. 
\end{definition}

We generalize this notion as follows.

\begin{definition}
Let $(M,d,\preceq)$ be an ordered metric space. We say that $d$ is left-monotone if  
$$x\preceq y \preceq z\Rightarrow d(x,y) \leq d(x,z)$$
for all $x,y,z\in M$.
\end{definition}

\begin{example}
\begin{enumerate}
\item[(a)] If $d$ is monotone then $d$ is left-monotone.
\item[(b)] Let $M= \mathbb{R} $ endowed with the usual distance $d(x,y)=|x-y|$ and the usual order $\leq$. Then, $d$ is monotone.
\item[(c)] Let $M= \mathbb{R}_+ $ endowed with the usual order $\leq$. Let the distance $d$ defined on  $\mathbb{R}_+$ by 
$$ d(x,y)= \begin{cases} x+y &\text{ if } x\neq y \\
0 &\text{ if } x= y.
\end{cases} $$  
$d$ is left-monotone. For all $x,y,z \in  \mathbb{R}_+$, such that
$ x\leq y\leq z $, we have $d(x,y)\leq d(x,z)$. 
But, if $y\neq x$ and $y\neq z$ we have $d(y,z) > d(x,z)$, i.e., $d$ is not $d$ monotone.
\end{enumerate}
\end{example}
%%%%%%%%%%%%%%%%%%%%%%%%%%%%%%%%%%%%%%%%%%%%%%%%%%%%%%%%%%%%%%%%%%%%%%%%%%%%%%%%%%%%%%%%%%%%%%%%%%%%%%%%%%%%%%%%%%%%%%%%%%%%%%%%%%%%%%%%%%%%%%
\section{Caristi-Penot fixed point theorem}\label{sec3}

In this section we prove some existence fixed point results for set-valued mappings in the spirit of Penot's formulation.

\begin{theorem}\label{Thm1}
Let  $(M,d,\preceq)$ be a partially ordered complete metric space where $d$ is left-monotone. Let $\phi : M \rightarrow \mathbb{R}_+$ be lower semi-continuous and $T : M \rightarrow \mathcal{K}(M)$ be a set-valued map such that for all $x\in M$, there exists $a\in Tx$, $x\preceq a$, such that 
$$ d(x,a) \leq \phi(x)- \inf(\phi(Tx \cap [x,a])).$$
Then, there exists $\overline{x} \in M$ such that $\overline{x} \in T\overline{x}$. 
\end{theorem}

\begin{proof}

Let $x\in M$. By hypothesis, there exists $a \in Tx$ such that $x \preceq a$ and 
$$ d(x,a) \leq \phi(x)- \inf(\phi(Tx \cap [x,a])).$$

Since $Tx \cap [x,a]$ is compact and $\phi$ is lower semicontinuous, there exists $f x \in T x \cap [x, a]$ such that $\phi(f x) = \inf \phi(T x\cap[x, a])$, so $d(x, a) \leq \phi(x)-\phi(f x).$ Since $d$ is left-monotone and $x \preceq f x \preceq a$, we get that $d(x, f x) \leq d(x, a)$. And then
$$d(x,fx)\leq \phi(x)-\phi(f x).$$

By Caristi's theorem, $f$ has a fixed point $x^*$. Since $f x^*\in T x^*$, $x^*$ is a fixed point of $T$.

\end{proof}

A similar result can be obtained without assuming that $\phi$ is lower semi-continuous and $T$ has compact values, but we need an additional condition, namely, the set-valued mapping $T$ has a closed graph. Recall that a set-valued $T:M\rightarrow 2^M$ has a closed graph if $G_T=\{(x,y)\in M^2 : y\in Tx\}$ is closed in $M^2$.

\begin{theorem}\label{THmCG}
Let  $(M,d,\preceq)$ be a partially ordered complete metric space where $d$ is left-monotone. Let $\phi : M \rightarrow \mathbb{R}_+$ be a function and $T : M \rightarrow 2^M$ a set-valued mapping with a closed graph such that for all $x\in M$, there exists $a\in Tx$, $x\preceq a$, such that 
$$ d(x,a) \leq \phi(x)- \inf(\phi(Tx \cap [x,a])).$$
Then, there exists $\overline{x} \in M$ such that $\overline{x} \in T\overline{x}$.
\end{theorem}

\begin{proof}
For $x_{0}\in M$, there exists $y_{0} \in Tx_{0}$ such that

$$ d(x_{0},y_{0}) \leq \phi(x_{0})- \inf(\phi(Tx_0 \cap [x_0,y_0])).$$

By definition of infimum,  there exists $x_{1} \in Tx_{0}$ such that $x_{0}\preceq x_{1}\preceq y_{0}$ and  
$$ d(x_{0},y_{0}) \leq \phi(x_{0})- \phi(x_{1})+\frac{1}{2}.$$

\medskip
Again, for this $x_1$, there exist $y_{1} \in Tx_{1}$ and $x_{2} \in Tx_{1}$ such that $x_{1}\preceq x_{2}\preceq y_{1}$ and  
$$ d(x_{1},y_{1}) \leq \phi(x_{1})- \phi(x_{2})+\frac{1}{2^2}.$$

\medskip
Continuing this process, we get two sequences $(x_{n})_{n\geq 0}$ and $(y_{n})_{n\geq 0}$ such that, for all $n\in \mathbb{N}$,  
$$ (x_{n+1},y_{n}) \in (Tx_{n})^{2}, \; x_{n}\preceq x_{n+1} \preceq y_{n} \text{ and } d(x_{n},y_{n}) \leq \phi(x_{n})- \phi(x_{n+1})+\frac{1}{2^{n+1}}.$$ 

For each $n\in\mathbb{N}$,

\[
\begin{array}{ccl}
  \sum_{k=0}^{n}d(x_{k},y_{k})&\leq &
\sum_{k=0}^{n}\left( \varphi(x_{k})-\varphi(x_{k+1}) \right)+\sum_{k=0}^n \frac{1}{2^{k+1}}\\
&= & \varphi(x_{0}) -\varphi(x_{n+1})+\sum_{k=0}^n \frac{1}{2^{k+1}}\\
&  \leq &\varphi(x_{0}) +\sum_{k=0}^\infty \frac{1}{2^{k+1}} < \infty.
\end{array}
\]
Then, the series $ \sum_{n\geq 0}d(x_{n},y_{n}) $ is convergent. 
As, $d$ is left-monotone and for all $n\in\mathbb{N}$, $x_{n}\preceq x_{n+1} \preceq y_{n}$, we have 
$$d(x_{n+1},x_{n}) \leq d(x_{n},y_{n}).$$ 

Therefore $ \sum_{n\geq 0}d(x_{n+1},x_{n}) $ is convergent. Hence, $(x_{n})_{n\geq 0}$ is a Cauchy sequence. As $M$ is a complete metric space, the sequence $(x_{n})_{n\geq 0}$  converges to $x^{*}= \lim_{n\to +\infty} x_{n}$. Since for all $n\in \mathbb{N}$, $x_{n+1} \in Tx_{n}$ and the mapping $T$ has a closed graph, we get $x^{*}\in Tx^{*}$. The proof is complete.
\end{proof}

%%%%%%%%%%%%%%%%%%%%%%%%%%%%%%%%%%%%%%%%%%%%%%%%%%%%%%%%%%%%%%%%%%%%%%%%%%%%%%%%%%%%%%%%%%%%%%%%%%%%%%%%%%%%%%%%%%%%%%%%%%%%%%%%%%%%%%%%%%%%%%
\section{Caristi-Penot fixed point theorem in reflexive Banach spaces}\label{sec4}

As reflexive spaces are complete spaces, we get the following result as a corollary of Theorem \ref{Thm1}

\begin{theorem}\label{ThmR}
Let $(M,\|\cdot\|,\preceq)$ be a reflexive normed space endowed with a partial order $\preceq$. Suppose that $\|\cdot\|$ is left-monotone. Let $\phi : M \rightarrow \mathbb{R}_+$ be lower semi-continuous and 
$T : M \rightarrow \mathcal{K}(M)$ be a set-valued mapping such that, for all $x\in M$, there exists $a\in Tx$, $x\preceq a$, such that 
$$  \|x-a\| \leq \phi(x)- \inf(\phi(Tx \cap [x,a])).$$

Then, there exists $x^* \in M$ such that $x^* \in Tx^* $. 
\end{theorem}

Let $(\mathcal{WK})(X)$ be the family of all nonempty weakly compact subsets of $X$. The aim purpose of this section is to give a similar result of Theorem \ref{Thm1} where the values of $T$ are weakly compact. The next technical lemmas will be useful to establish Caristi-Penot theorem in reflexive normed spaces. Note that we assume that the order intervals are closed and convex. 

\begin{lemma}\cite[Lemma 1]{ref15}\label{L41} 
Let $X$ be a normed space endowed with a partial order $\preceq$. Assume that $(x_n)_{n\in \mathbb{N}}$ and $(y_n)_{n\in \mathbb{N}}$ are two sequences on $X$ which are weakly convergent to $x$ and $y$ respectively and $x_n \preceq y_n$ for any $n \in \mathbb{N}$, then
$$ x\preceq y.$$
\end{lemma}

\begin{proof} 
Note that the positive sequence $(y_n - x_n)_n$ converges weakly to $y-x$.
Since closed convex subsets are also weakly closed, the positive cone is weakly closed and so we conclude that $y-x$ is positive.
\end{proof}

\begin{lemma}\cite[Proposition 3.3]{ref14} \label{L42}
Let $(x_n)_{n\in \mathbb{N}}$ be a bounded monotone increasing or decreasing sequence in $M$, and assume that $M$ is reflexive. Then, $(x_n)_{n\in \mathbb{N}}$ is weakly convergent.
\end{lemma}  

\begin{lemma} 
Let $(x_n)_{n\in \mathbb{N}}$ be a bounded sequence in a reflexive Banach space $M$. Then $(x_n)_{n\in \mathbb{N}}$ has a weakly convergent subsequence.
\end{lemma}

\begin{definition}
Let $M$ be a nonempty subset in normed space $X$. We say $\phi : M \rightarrow \mathbb{R}_+$ is lower weakly semi-continuous if for any sequence $(x_{n})_{n\geq 0}$  of $M$, converging weakly to an element $x$ of $M$, ${\displaystyle \phi(x)\leq \liminf_{n\to +\infty}\phi(x_{n})}$.
\end{definition}

Now, we are able to give the main result of this section.

\begin{theorem}\label{Thm88}
Let $(X,\|.\|,\preceq)$ be a reflexive space endowed with a partial order $\preceq$ and $M$ be a nonempty bounded closed subset of $X$. Suppose that $d$ is left-monotone. Let $\phi : M \rightarrow \mathbb{R}_+$ be a lower weakly semi-continuous, and 
$T : M \rightarrow (\mathcal{WK})(M)$ be a set-valued mapping such that, for all $x\in M$, there exists $a\in Tx$, $x\preceq a$, such that 
$$  \|x-a\| \leq \phi(x)- \inf(\phi(Tx \cap [x,a])).$$

Then, there exists $x^* \in M$ such that $x^* \in Tx^* $. 
\end{theorem}

\begin{proof}
Let $x\in M$. As $Tx$ is weakly compact and $\phi$ is lower weakly semi-continuous function, there exist $a,y \in Tx$ such that $ \inf(\phi(Tx \cap [x,a]))= \phi(y)$. So 
$$ \|x-a\| \leq \phi(x)- \phi(y).$$
Define the relation $<_{\phi}$ on $M$ by 
\begin{equation}\label{eq1}
  y <_{\phi} x \Leftrightarrow [ \exists a\in M, \; y\in [x,a] \text{ and } \|x-a\| \leq \phi(x)- \phi(y) ],\tag{PO}
\end{equation}
for all $x,y\in M$. Then, the relation $<_{\phi}$ is a partial order on $M$. Indeed,
\begin{itemize}
\item The relation $<_{\phi}$ is reflexive $x<_{\phi} x$, take $x=a$,
\item $<_{\phi}$ is antisymmetric, if $x<_{\phi}y$ and $y<_{\phi}x$, then there exist $a,b\in M$ such that $y\in [x,a]$ and $x\in [y,b]$, so $x=y$, 
\item $<_{\phi}$ is transitive, if $y<_{\phi}x$ and $x<_{\phi}z$, then there exist $a,b\in M$ 
$$
y\in [x,a] \text{ and } \|x-a\| \leq \phi(x)- \phi(y)
$$
and
$$
x\in [z,b] \text{ and } \|z-b\| \leq \phi(z)- \phi(x).
$$
Since $z\preceq x\preceq b$ and $x\preceq y$, so $y\in [z,a]$. By the left-monotonicity we have $\|z-x\| \leq \|z-b\|$. Hence,
\begin{align*}
\|z-a\| & \leq \|z-x\| + \|x-a\| \\
& \leq   \|z-b\|+\|x-a\| \\
& \leq \phi(z)- \phi(y). 
\end{align*}
Thus, $y<_{\phi} z$.
\end{itemize} 

We show that any decreasing chain in $(M, <_{\phi})$ has a lower bound. Let $(x_{\alpha})_{\alpha \in \Gamma}$ be a decreasing chain in $(M, <_{\phi})$ and let $(\alpha_{n})$ be an increasing sequence of elements from $\Gamma$ such that 
$$ \lim_{n\to +\infty} \phi(x_{\alpha_{n}}) = \inf\{\phi(x_{\alpha}) :  \alpha \in \Gamma \}.$$

\begin{itemize}
  \item[\textbf{Step 1.}]
As $(x_{\alpha_{n}})_{n\in \mathbb{N}}$ is a bounded increasing sequence w.r.t. $\preceq$ in reflexive space $X$, so by Lemma \ref{L41} and Lemma \ref{L42}, $(x_{\alpha_{n}})_{n\geq \mathbb{N}}$ converges weakly to an element $x\in M$ and $x_{\alpha_{n}} \preceq x$, for all $n\in \mathbb{N}$. Let $n,m\in \mathbb{N}$ such that $m>n$. Then
$$ \|x_{\alpha_{n}}-x_{\alpha_{m}}\|\leq \sum_{k=n}^{m-1} \|x_{\alpha_{k}}-x_{\alpha_{k+1}}\|.$$

Since $$ x_{\alpha_{n}}\preceq x_{\alpha_{n+1}} \preceq \cdots\preceq x_{\alpha_{k}} \preceq x_{\alpha_{k+1}}\preceq\cdots \preceq x_{\alpha_{m}}$$
so, for each integer $k \in [n,m)$, there exists $a_k\in Tx_{\alpha_k}$ such that
$$\|x_{\alpha_{k}}-a_{k}\|\leq \phi(x_{\alpha_{k}}) - \phi(x_{\alpha_{k+1}}),$$ 
and by the left monotony we get
$$\|x_{\alpha_{k}}-x_{\alpha_{k+1}}\| \leq \phi(x_{\alpha_{k}}) - \phi(x_{\alpha_{k+1}}).$$

Hence,
$$ \|x_{\alpha_{n}}-x_{\alpha_{m}}\|\leq \sum_{k=n}^{m-1}(\phi(x_{\alpha_{k}}) - \phi(x_{\alpha_{k+1}})) = \phi(x_{\alpha_{n}}) - \phi(x_{\alpha_{m}}).$$ 
Letting $m\rightarrow +\infty$, we obtain
\begin{align*} \|x_{\alpha_{n}}-x\| & \leq \liminf_{m\to +\infty} \|x_{\alpha_{n}}-x_{\alpha_{m}}\| \\
& \leq  \phi(x_{\alpha_{n}}) - \lim_{m\to +\infty} \phi(x_{\alpha_{m}})\\
& \leq \phi(x_{\alpha_{n}}) - \phi(x),\quad\text{because}\;\phi\; \text{is lower weakly semi-continuous}. 
\end{align*}
Thus, for $a=x$ and all $n\in \mathbb{N}$, 
$$ x_{\alpha_{n}} \preceq x \preceq a \text{ and } 
\|x_{\alpha_{n}}-x\| \leq  \phi(x_{\alpha_{n}}) - \phi(x),$$
i.e., $x$ is a lower bound for $(x_{\alpha_{n}})_{n\in \mathbb{N}}$ in $(M,<_\phi)$.

\item[\textbf{Step 2.}] In order to see that $x$ is a lower bound for $(x_{\alpha})_{\alpha \in \Gamma}$. 

\begin{itemize}
  \item[Case 1.]  Suppose there exists $\beta \in \Gamma$ such that $x_{\beta} <_{\phi} x_{\alpha_{n}}$ for all $n \in \mathbb{N}$. Then
 $$\phi(x_{\beta}) = \inf\{\phi(x_{\alpha}) : \alpha \in \Gamma \}.$$
Since $x_{\beta} <_{\phi} x_{\alpha_{n}}$ for all $n \in \mathbb{N}$, there exists $c_{n} \in M$ such that 
$$  x_{\alpha_{n}} \preceq  x_{\beta} \preceq c_{n} \text{ and } \| x_{\alpha_{n}}-c_{n}\| \leq \phi(x_{\alpha_{n}}) - \phi(x_{\beta}).$$ 
Letting $n\rightarrow +\infty$, we get $\lim_{n\to +\infty} \|x_{\alpha_{n}} -c_{n}\|=0$. As $(c_n)_n$ is a bounded sequence in $M$, there exists a subsequence $(c_{\psi(n)})_{n\in \mathbb{N}}$ of $(c_n)_n$ which converges weakly to an element $c\in M$. Then, 
$$ \|x-c\| \leq \liminf_{n\to +\infty} \|x_{\psi(\alpha_{n})}-c_{\psi(n)}\|=0.$$ 
Thus, $x=c$. Since, $x_{\alpha_{\psi(n)}} \preceq  x_{\beta}\preceq c_{\psi(n)}$ for all $n\in \mathbb{N}$, we get $x= x_{\beta} $. 

\item[Case 2.] Suppose for any $\alpha \in \Gamma$, there exists $n \geq  0$ such that $ x_{\alpha_{n}} <_{\phi} x_{\alpha}$, i.e., there exists $z_{n} \in M$ satisfies
$$
x_{\alpha} \preceq x_{\alpha_{n}} \preceq z_{n}
$$ 
and
$$
\|x_{\alpha}-z_{n}\| \leq  \phi(x_{\alpha}) - \phi(x_{\alpha_{n}}).
$$
   
So, $x_{\alpha}\preceq x_{\alpha_{n}} \preceq x$  and   
$$
\begin{array}{ccl}
  \|x_{\alpha}-x_{\alpha_{n}}\| &\leq& \|x_{\alpha}-z_{n}\|\\
   &\leq & \phi(x_{\alpha}) - \phi(x_{\alpha_{n}})\\
   &\leq &\phi(x_{\alpha}) - \inf\{\phi(x_{\alpha}) :  \alpha \in \Gamma \} \\
   &\leq & \phi(x_{\alpha}) - \phi(x) .
\end{array}
$$   
Thus, $x <_{\phi} x_{\alpha}$,  i.e., $x$ is a lower bound of $(x_{\alpha})_{\alpha \in \Gamma}$. 
\end{itemize}

\item[\textbf{Step 3.}]
Zorn's lemma will therefore imply that $(M, <_{\phi})$ has minimal elements, that we denote by $\overline{x} $. For this $\overline{x} $, there exist $\overline{a} \in T\overline{x}$ and $\overline{y}\in T\overline{x} \cap [\overline{x},\overline{a}]$ such that 
$$\inf(\phi(T\overline{x} \cap [\overline{x},\overline{a}]))= \phi(\overline{y}).$$ 
So 
$$ \overline{x}\preceq  \overline{y}\preceq \overline{a}\quad \text{and}\quad \|\overline{x}-\overline{a}\| \leq \phi(\overline{x})- \phi(\overline{y}) .$$
Hence, $ \overline{y} <_{\phi}  \overline{x}$. Since, $\overline{x}$ is minimal elements of $(M,<_{\phi})$, we get $\overline{x}=\overline{y} \in T\overline{x}$. The proof is complete.
\end{itemize}
\end{proof}

%%%%%%%%%%%%%%%%%%%%%%%%%%%%%%%%%%%%%%%%%%%%%%%%%%%%%%%%
 \section*{Acknowledgements} The authors would like to thank the anonymous referees for their careful reading of our manuscript and their many thoughtful comments and suggestions.
%%%%%%%%%%%%%%%%%%%%%%%%%%%%%%%%%%%%%%%%%%%%%%%%%%%%%%%%%%%%Biblio

\end{document}